\documentclass[reqno,12pt,twoside]{article}
\usepackage{fullpage}
\usepackage{amssymb}
\usepackage{amsmath}
\usepackage{amsthm}
\numberwithin{equation}{section}
\newtheorem{theorem}{Theorem}

\newtheorem{remark}{Remark}

\newtheorem{corollary}{Corollary}

\newtheorem{lemma}{Lemma}
\newcommand{\Nzero}{\mathbb N_0}
\newcommand{\trbinom}[2]{\left\langle\binom{#1}{#2}\right\rangle}
\newcommand{\pos}[1]{\left(#1\right)_{+}}
\newcommand{\card}[1]{\left|#1\right|}

\begin{document}

\title{\Large On the size of $h$-fold sumsets \footnote{The first author is supported by the National Natural Science Foundation of China (Grant No.
12301003), the Anhui Provincial Natural Science Foundation (Grant No. 2308085QA02) and
the University Natural Science Research Project of Anhui Province (Grant No. 2022AH050171). The second author is supported by the National Natural Science Foundation of China (Grant No.
12371005).}}

\author{\large Shi-Qiang Chen$^{a}$,
\large ~~~Quan-Hui Yang$^{b,}$\thanks{Corresponding author. E-mail addresses: csq20180327@163.com (S.-Q. Chen), yangquanhui01@163.com (Q.-H. Yang)}}
\date{} \maketitle
 \vskip -3cm
\begin{center}
\vskip -1cm { \small
\begin{center}
 $a$. School of Mathematics and Statistics,
\end{center}
\begin{center}
Anhui Normal University, Wuhu 241002, P.R. China
\end{center}}
\vskip -1cm { \small
\begin{center}
 $b$. School of
Mathematical Sciences and Ministry of Education Key Laboratory for NSLSCS,
\end{center}
\begin{center}
Nanjing Normal University, Nanjing 210023,
P.R. China
\end{center}}
\end{center}

{\bf Abstract.} Let $h$ be a positive integer, and let $A$ be a finite set of integers. We derive an exact formula for $|hA|$. Furthermore, let
$A=\{0,1,\ldots,s,a,b\}$, $1\leq s<a<b$, and write $b=qa+r$ with $0\leq r<a$. By using generating function, we prove that $|hA|$ equals a definite explicit formula expressed in terms of certain truncated binomial coefficients for all positive integers $h$ if and only if $r=0$ or $qs+r\geq a$. This generalizes a result of Nathanson.

\medskip
\noindent{\bf Keywords:} $h$-fold sumset; sumset size;
generating function; truncated binomial coefficient

\medskip
\noindent 2020 {\it Mathematics Subject Classification}: 11B13.

\section{Introduction}
Let $\mathbb{N}$ denote the set of positive integers,
$\mathbb{N}_0=\mathbb{N}\cup\{0\}$, and let $\mathbb{Z}$ denote the set
of integers.  For integers $u$ and $v$ with $u\leq v$, we define
\[
[u,v]=\{c\in\mathbb{Z}:u\leq c\leq v\}.
\]
If $A\subseteq\mathbb{Z}$ is finite and nonempty and $h\in\mathbb{N}$,
the $h$-fold sumset of $A$ is
\[
hA=\{a_1+\cdots+a_h:a_i\in A\text{ for }1\leq i\leq h\}.
\]
The problem of determining $|hA|$ is a basic topic in additive number
theory.  It is closely related to structural and extremal questions for
higher sumsets and to Frobenius-type representation problems
\cite{EliahouKervaire,GranvilleShakan,GranvilleWalker,
GyarmatiMatolcsiRuzsa,Lev1996,Lev,Nathanson1972}.

For $h,k\in\mathbb{N}$, define the sumset-size set
\[
\mathcal{R}(h,k)=\{|hA|:A\subseteq\mathbb{Z},\ |A|=k\}.
\]
Injective affine transformations do not change the size of an
$h$-fold sumset: for a set $B\subseteq \mathbb{Z}$, if there are integers $u \neq 0$ and $v$ such that
\[B=u\ast A+v=\{ua +v : a \in A\},\]
then
$hB=u\ast(hA)+hv$ and hence $|hB|=|hA|$.  Consequently, after translating,
reflecting if necessary, and dividing all elements by their common
divisor, a $k$-element set may be normalized as
\[
A=\{0=a_0<a_1<\cdots<a_{k-1}\},
\qquad \gcd(a_1,\ldots,a_{k-1})=1.
\]

The classical extremal bounds are
\[
h(k-1)+1\leq |hA|\leq\binom{h+k-1}{k-1}.
\]
For $h\geq2$, the lower bound is attained exactly by arithmetic
progressions.  The upper bound is attained when every sum has a unique
representation by a nondecreasing $h$-tuple from $A$.  Exact information
between these extremes is considerably more delicate.

In 2025, Nathanson \cite{Nathanson2025} determined the complete
range when $h=2$:\\
{\bf Theorem A.} For all positive integers $k$,
\[\mathcal{R}(2,k)= \left[2k - 1,\frac{k^2 + k}{2}\right].\]
Nathanson \cite{Nathanson2025} also determined the following three element case.\\
{\bf Theorem B.} For all positive integers $h$,
\[\mathcal{R}(h,3) = \left\{ \dbinom{h+2}{2} - \dbinom{\ell}{2} : \ell \in [1,h] \right\}.\]

In 2025, Nathanson \cite{Nathanson2025} also proposed the following problem.\\
\noindent \textbf{Problem 1.} The structure of sumsets of sets of four integers is still not
understood. Compute the set of sumset sizes $\mathcal{R}(h,4)$. What is the cardinality of this
set?

For four element sets, the upper endpoint is the tetrahedral number
$\binom{h+3}{3}$, but the values below it already display a much richer
arithmetic structure; see
\cite{NathansonExplicit,NathansonTriangular,NathansonTetrahedral,OBryant}.

In particular, Nathanson \cite{NathansonTetrahedral} exhibited the
following family.\\
{\bf Theorem C.}  Let $h \geq 1$. For all $i_0 \in [0, h-1]$, let
\[
p = 1 + (i_0 - 1)(h + 1)
\]
and
\[
c = h^2 + h + 1 - p = (h + 1 - i_0)(h + 1).
\]
The set
\[
A = \{0, 1, h + 1, c\}
\]
satisfies $|A| = 4$ and
\[
|hA| = \dbinom{h+3}{3} - \dbinom{i_0+2}{3}.
\]

Nathanson \cite{NathansonTetrahedral} also proposed the following two problems.

\noindent \textbf{Problem 2.} It is of interest to compute, for all $h \geq 3$ and all $p \in [0, h^2 - 1]$, the sumset sizes of the sets
\[
A = \{0, 1, h + 1, h^2 + h + 1 - p\}.
\]

\noindent \textbf{Problem 3.} For all $h \geq 3$, compute the set of sumset sizes of the sets
\[
A = \{0, 1, a, b\}
\]
for $2 \leq a \leq h$ and $a + 1 \leq b \leq ha + 1$.

We treat these questions in a common framework that is valid more
generally for every finite normalized set.  The basic tool is the
shortest representation function.  Let
\[
A=\{0=a_0<a_1<\cdots<a_k=m\},
\qquad \gcd(a_1,\ldots,a_k)=1.
\]
For $n\in\mathbb{N}_0$, define
\[
\ell_A(n)=
\min\left\{\sum_{j=1}^k x_j:
n=\sum_{j=1}^k a_jx_j,\ x_j\in\mathbb{N}_0\right\},
\]
with $\ell_A(n)=+\infty$ when $n$ is not representable. Hence,
\[
n\in hA\quad\Longleftrightarrow\quad \ell_A(n)\leq h.
\]

For an integer $0\leq v<m$, let
\[
C_v=
\left\{
(x_1,\ldots,x_{k-1})\in\{0,\ldots,m-1\}^{k-1}:
\sum_{j=1}^{k-1}a_jx_j\equiv v\pmod m
\right\}.
\]
From $\gcd(a_1,\dots,a_{k-1},m)=1$, it follows that each $C_v$ is nonempty.
For $x\in C_v$, define
\[
L_v(x)
=\frac{\sum_{j=1}^{k-1}a_jx_j-v}{m},
\qquad
D_v(x)
=\sum_{j=1}^{k-1}x_j-L_v(x).
\]
The numerator defining \(L_v(x)\) is divisible by \(m\), so both
\(L_v(x)\) and \(D_v(x)\) are integers.

For an integer $0\leq v<m$ and $t\in\Nzero$, define
\[
M_t(v)=\min_{\substack{x\in C_v\\L_v(x)\leq t}}D_v(x).
\]
Set $M_t(v)=+\infty$ if the set over which the minimum is taken is empty.
Starting from its first finite value, $M_t(v)$ is nonincreasing and
eventually constant.  Write its finite strict records as
\[
M_t(v)=d_{v,i}
\qquad
\bigl(u_{v,i}\leq t<u_{v,i+1}\bigr),
\qquad 1\leq i\leq r_v,
\]
where
\[
u_{v,1}<\cdots<u_{v,r_v},
\qquad
d_{v,1}>\cdots>d_{v,r_v},
\qquad
u_{v,r_v+1}=+\infty.
\]
Thus $M_t(v)=+\infty$ for $0\leq t<u_{v,1}$.

For an integer $x$ and $j\in\Nzero$, define $(x)_+=\max(x,0)$ and define the truncated binomial
coefficient by
\[
\trbinom{x}{j}
=
\begin{cases}
\binom{x}{j},&x\geq j,\\
0,&x<j.
\end{cases}
\]
In particular, $\trbinom{x}{1}=\max(x,0)$.

Our first result gives the exact size of $hA$.

\begin{theorem}\label{thm1}
Let $h\in\mathbb{N}$ and let
$A=\{0=a_0<a_1<\cdots<a_k=m\}$ with
$\gcd(a_1,\ldots,a_k)=1$.  Then
\begin{align*}
\card{hA}
={}&\sum_{v=0}^{m-1}\Biggl\{
\sum_{i=1}^{r_v-1}
\trbinom{h-u_{v,i}-d_{v,i}+1}{1}
\\
&\qquad
-\sum_{i=1}^{r_v-1}
\trbinom{h-u_{v,i+1}-d_{v,i}+1}{1}
\\
&\qquad
+\trbinom{h-u_{v,r_v}-d_{v,r_v}+1}{1}
\Biggr\}.
\end{align*}
\end{theorem}

Theorem~\ref{thm1} holds for every finite normalized set, so it works for any choice of parameters in Problems 2 and 3. However, the record values depend on the elements of $A$, and usually there is no simple closed form. For this reason, we only study sets of the form $\{0,1,\dots,s,a,b\}$, and give the necessary and sufficient condition for the size of the $h$-fold sumset of this set to admit a definite explicit formula in terms of certain truncated binomial coefficients.

For $s\in\mathbb{N}$ and integer $N\geq1$, put
\begin{equation}\label{eq11}
L_{s,N}=\left\lfloor\frac{N-1}{s}\right\rfloor,
\qquad
\tau_{s,N}=N-1-sL_{s,N},
\end{equation}
and define
\[
P_{s,N}(z)
=1+(s-1)z+(\tau_{s,N}-s)z^{L_{s,N}+1}
-\tau_{s,N}z^{L_{s,N}+2},
\qquad
P_{s,0}(z)=0.
\]
Write
\[
P_{s,N}(z)=\sum_e p_{N,e}z^e.
\]
\begin{theorem}\label{thm2}
Let $s,a,b$ be integers with $1\leq s<a<b$, and let
\[
A=\{0,1,\ldots,s,a,b\}.
\]
Write $b=qa+r$, where $0\leq r<a$.  Then
\begin{align}\label{eq01}
\card{hA}
={}&
\sum_e p_{a,e}
\left\{
\trbinom{h-e+3}{3}
-\trbinom{h-q-e+3}{3}
\right\}
\notag\\
&+\sum_e p_{r,e}
\trbinom{h-q-e+2}{2}
\end{align}
holds for every $h\in\mathbb{N}$ if and only if $r=0$ or $qs+r\geq a$.
\end{theorem}

Taking $s=1$ gives the four-element case in a particularly transparent
six-term form.

\begin{corollary}\label{coro1}
Let $a, b$ be integers with $1<a<b$. Write $b=qa+r$, where $0\leq r<a$.  Then
\begin{align}\label{eq02}
\card{h\{0,1,a,b\}}
={}&\trbinom{h+3}{3}
-\trbinom{h-q+3}{3}
-\trbinom{h-a+3}{3}
+\trbinom{h-a-q+3}{3}\notag\\
&+\trbinom{h-q+2}{2}
-\trbinom{h-q-r+2}{2}
\end{align}
holds for every $h\in\mathbb{N}$ if and only if $r=0$ or $q+r\geq a$.
\end{corollary}
\begin{proof}
Take $s=1$ in Theorem \ref{thm2}.  Then
\[
P_{1,a}(z)=1-z^a,\qquad P_{1,r}(z)=1-z^r.
\]
Thus \eqref{eq02} follows.
\end{proof}

\begin{remark}
Corollary \ref{coro1} contains Nathanson's result:
take $a=h+1$, $r=0$, and $q=h+1-i_0$.  Outside the regular range
$r=0$ or $q+r\geq a$, the short formula \eqref{eq02} fails in general,
but Theorem \ref{thm1} continues to give the exact value.
\end{remark}

\section{Proof of Theorem \ref{thm1}}
Throughout this section, $A$ is as in Theorem \ref{thm1}, with
$a_k=m$.  We begin with three lemmas.

\begin{lemma}\label{lema11}Let $n$ be a nonnegative integer.
Suppose that
\[
n=\sum_{j=1}^{k-1}a_jx_j+mw,
\qquad
\sum_{j=1}^{k-1}x_j+w=\ell_A(n),\qquad x_1,\dots,x_{k-1},w\in\mathbb{N}_0.
\]
Then $0\leq x_j<m$ for any $1\leq j\leq k-1$.
\end{lemma}
\begin{proof}
Assume that there exists $j_0\in[1,k-1]$ such that \(x_{j_0}\geq m\). Then
\[n=\sum_{\substack{j\in[1,k-1],\\j\neq j_0}}a_jx_j+(x_{j_0}-m)a_{j_0}+m(w+a_{j_0}).\]
Noting that $a_{j_0}<m$, we have
\[\sum_{\substack{j\in[1,k-1],\\j\neq j_0}}x_j+(x_{j_0}-m)+(w+a_{j_0})<\sum_{j=1}^{k-1}x_j+w=\ell_A(n),\]
a contradiction. Thus, \(0\leq x_j<m\) for any $j\in[1,k-1]$.

This completes the proof of Lemma \ref{lema11}.

\end{proof}

\begin{lemma}\label{lema1}
For any $t\in\Nzero$ and any integer $0\leq v<m$, we have
\begin{equation}\label{eq21}
\ell_A(tm+v)
=t+M_t(v).
\end{equation}
\end{lemma}
\begin{proof}
For any $t\in\Nzero$ and any integer $0\leq v<m$, by the definition of $\ell_A(tm+v)$, we have
\[
\ell_A(tm+v)=\min\left\{w+\sum_{j=1}^{k-1}x_j:
tm+v=\sum_{j=1}^{k-1}a_jx_j+mw,\
x\in\Nzero^{k-1},\ w\in\Nzero\right\}.
\]
It follows from the definition of $C_v$ and Lemma \ref{lema11} that
\[\ell_A(tm+v)=\min\left\{w+\sum_{j=1}^{k-1}x_j:tm+v=\sum_{j=1}^{k-1}a_jx_j+mw, ~x\in C_v,w\geq0\right\}.\]
By the definitions of $L_v(x)$ and $D_v(x)$, we have
\begin{align*}
\ell_A(tm+v)
=&\min_{\substack{x\in C_v,\\L_v(x)\leq t}}\left\{t-L_v(x)+\sum_{j=1}^{k-1}x_j\right\}\nonumber\\
=&\min_{\substack{x\in C_v,\\L_v(x)\leq t}}\{t+D_v(x)\}\nonumber\\
=&t+M_t(v).
\end{align*}

This completes the proof of Lemma \ref{lema1}.
\end{proof}

\begin{lemma}\label{lema2}
Let $u, U$ be integers with $u<U$, and let $d,h\in\mathbb Z$.
Then
\begin{equation}\label{eq221}
|\{t\in\mathbb Z:u\leq t<U,\ t+d\leq h\}|
=\pos{h-u-d+1}-\pos{h-U-d+1}
\end{equation}
and
\begin{equation}\label{eq222}
|\{t\in\mathbb Z:t\geq u,\ t+d\leq h\}|
=\pos{h-u-d+1}.
\end{equation}
\end{lemma}
\begin{proof}
If \(h-d<u\), then
\[|\{t\in\mathbb Z:u\leq t<U,\ t+d\leq h\}|=\pos{h-u-d+1}-\pos{h-U-d+1}=0.\]
If \(u\leq h-d<U\), then
\[|\{t\in\mathbb Z:u\leq t<U,\ t+d\leq h\}|=\pos{h-u-d+1}-\pos{h-U-d+1}=h-d-u+1.\]
 If \(h-d\geq U\), then
\[|\{t\in\mathbb Z:u\leq t<U,\ t+d\leq h\}|=\pos{h-u-d+1}-\pos{h-U-d+1}=U-u.\]
Hence, \eqref{eq221} is true. \eqref{eq222} is immediate from counting arguments.

This completes the proof of Lemma \ref{lema2}.
\end{proof}
\begin{proof}[Proof of Theorem \ref{thm1}]
For any \(n\in\Nzero\), $n$ has a unique expression
\[
n=tm+v,\qquad t\in\Nzero,\quad 0\leq v<m.
\]
By Lemma \ref{lema1}, we have
\[
\ell_A(tm+v)=t+M_t(v).
\]
It follows that
\begin{equation}\label{eq22}
tm+v\in hA
\quad\Longleftrightarrow\quad\ell_A(tm+v)\leq h\quad\Longleftrightarrow\quad
 t+M_t(v)\leq h.
\end{equation}

Fix an integer \(v\in[0,m-1]\).  If \(t<u_{v,1}\), then \(M_t(v)=+\infty\), and so there are no admissible integers $t$ from \eqref{eq22}.  Since
\[
u_{v,i}\leq t<u_{v,i+1},
\qquad
M_t(v)=d_{v,i}
\]
for $1\leq i\leq r_v$,
it follows from \eqref{eq22} that for any integer $t$ with $u_{v,i}\leq t<u_{v,i+1}$,
\begin{equation}\label{eq23}
tm+v\in hA
\quad\Longleftrightarrow\quad t+d_{v,i}\leq h.
\end{equation}
If \(i<r_v\), take
\[
u=u_{v,i},\qquad U=u_{v,i+1},\qquad d=d_{v,i}
\]
in \eqref{eq221}.  The number of integers $t$ satisfying \eqref{eq23} is
\[
\trbinom{h-u_{v,i}-d_{v,i}+1}{1}
-
\trbinom{h-u_{v,i+1}-d_{v,i}+1}{1}.
\]
If \(i=r_v\), the record segment has no finite right endpoint, and
\eqref{eq222} gives only
\[
\trbinom{h-u_{v,r_v}-d_{v,r_v}+1}{1}.
\]
Hence,
\begin{align*}
\card{hA}
={}&\sum_{v=0}^{m-1}\Biggl\{
\sum_{i=1}^{r_v-1}
\trbinom{h-u_{v,i}-d_{v,i}+1}{1}
\\
&\qquad
-\sum_{i=1}^{r_v-1}
\trbinom{h-u_{v,i+1}-d_{v,i}+1}{1}
\\
&\qquad
+\trbinom{h-u_{v,r_v}-d_{v,r_v}+1}{1}
\Biggr\}.
\end{align*}

This completes the proof of Theorem \ref{thm1}.
\end{proof}
\section{Proof of Theorem \ref{thm2}}
For an integer $y$ and a positive integer $m$, let $[y]_m$ denote the least nonnegative residue of $y$ modulo $m$. We now prepare the shortest representation and generating function
lemmas used to prove Theorem \ref{thm2}. For convenience, recall that,
for $s\in\mathbb{N}$ and integer $N\geq1$,
\[
L_{s,N}=\left\lfloor\frac{N-1}{s}\right\rfloor,
\qquad
\tau_{s,N}=N-1-sL_{s,N}=[N-1]_s.
\]
\begin{lemma}
\label{lema31} Let \(B\subseteq \Nzero\) be a finite set containing \(0\) and \(1\), let $b$ be a positive integer with
\(b>\max B\), and put \(A=B\cup\{b\}\). For any integer \(0\leq v<b\) and any \(t\in\Nzero\), we have
\begin{equation}
\ell_A(tb+v)
=t+\min_{0\leq j\leq t}\{\ell_B(v+jb)-j\}.
\label{eq:quotientminimum}
\end{equation}
\end{lemma}
\begin{proof}Noting that \(A=B\cup\{b\}\), we have
\begin{align*}
\ell_A(tb+v)
=&\min_{0\leq y\leq t}\{y+\ell_B(tb+v-yb)\}\nonumber\\
=&\min_{0\leq j\leq t}\{t-j+\ell_B(v+jb)\}\nonumber\\
=&t+\min_{0\leq j\leq t}\{\ell_B(v+jb)-j\}
\end{align*}
for any integer \(0\leq v<b\) and any \(t\in\Nzero\).

This completes the proof of Lemma \ref{lema31}.
\end{proof}

\begin{lemma}\label{lema32} Let \(B\subseteq \Nzero\) be a finite set containing \(0\) and \(1\), let $b$ be a positive integer with
\(b>\max B\), and put \(A=B\cup\{b\}\). If
\(\ell_B(tb+v)\geq t+\ell_B(v)\) for any integer
\(0\leq v<b\) and any integer \(t\geq0\), then
\begin{equation}\label{eq31}
\ell_A(tb+v)=t+\ell_B(v)
\end{equation}
for any integer \(0\leq v<b\) and any integer \(t\geq0\).
\end{lemma}
\begin{proof}
For any integer \(0\leq v<b\) and any integer \(t\geq0\), by Lemma \ref{lema31}, we have
\[
\ell_A(tb+v)
=t+\min_{0\leq j\leq t}\{\ell_B(v+jb)-j\}
\leq t+\ell_B(v).
\]
Combining the hypotheses, we get
\[
\ell_A(tb+v)=t+\ell_B(v)
\]
for any integer \(0\leq v<b\) and any integer \(t\geq0\).

This completes the proof of Lemma \ref{lema32}.
\end{proof}

\begin{lemma}\label{lema33}
Let $s, a$ be integers with $1\leq s<a$, and let
\[
B=\{0,1,\ldots,s,a\}.
\]
For any \(t\in\Nzero\) and any integer \(0\leq x<a\), we have
\[
\ell_{B}(ta+x)
=t+\left\lceil\frac{x}{s}\right\rceil.
\]
\end{lemma}
\begin{proof}We define $B'=\{0,1,\ldots,s\}$. For any \(t\in\Nzero\) and any integer \(0\leq x<a\), by Lemma \ref{lema31}, we have
\begin{equation}\label{eq32}
\ell_{B}(ta+x)=t+\min_{0\leq j\leq t}\{\ell_{B'}(x+ja)-j\}.
\end{equation}
By the definition of $B'$, we have
\[\ell_{B'}(x+ja)=\left\lceil\frac{x+ja}{s}\right\rceil.\]
It follows from \eqref{eq32} that
\[\ell_{B}(ta+x)=t+\min_{0\leq j\leq t}\left\lceil\frac{x+j(a-s)}{s}\right\rceil=t+\left\lceil\frac{x}{s}\right\rceil.\]

This completes the proof of Lemma \ref{lema33}.
\end{proof}

\begin{lemma}\label{lema335}
Let $s, N$ be positive integers, and let $L_{s,N}$ and $\tau_{s,N}$ be as defined in \eqref{eq11}.  Then
\begin{equation}
F_{s,N}(z)
:=\sum_{x=0}^{N-1}z^{\lceil x/s\rceil}
=1+s\sum_{j=1}^{L_{s,N}}z^j+\tau_{s,N}z^{L_{s,N}+1}
=\frac{P_{s,N}(z)}{1-z}.
\label{eq:prefixpolynomial}
\end{equation}
\end{lemma}
\begin{proof}Since
\[N-1=L_{s,N}s+\tau_{s,N}, ~~~0\leq \tau_{s,N}<s,\]
 it follows that
\begin{align*}
F_{s,N}(z)
=&\sum_{x=0}^{N-1}z^{\lceil x/s\rceil}\nonumber\\
=&\sum_{j=0}^{L_{s,N}-1}\sum_{t=0}^{s-1}z^{\lceil (js+t)/s\rceil}+z^{L_{s,N}}+\sum_{t=1}^{\tau_{s,N}}z^{\lceil (L_{s,N}s+t)/s\rceil}\nonumber\\
=&\sum_{j=0}^{L_{s,N}-1}\sum_{t=0}^{s-1}z^{j+\lceil t/s\rceil}+z^{L_{s,N}}+\sum_{t=1}^{\tau_{s,N}}z^{L_{s,N}+\lceil t/s\rceil}\nonumber\\
=&\sum_{j=0}^{L_{s,N}-1}(\sum_{t=1}^{s-1}z^{j+1}+z^j)+z^{L_{s,N}}+\sum_{t=1}^{\tau_{s,N}}z^{L_{s,N}+1}\nonumber\\
=&(s-1)\sum_{j=0}^{L_{s,N}-1}z^{j+1}+\sum_{j=0}^{L_{s,N}-1}z^j+z^{L_{s,N}}+\sum_{t=1}^{\tau_{s,N}}z^{L_{s,N}+1}\nonumber\\
=&(s-1)\sum_{j=1}^{L_{s,N}}z^{j}+\sum_{j=0}^{L_{s,N}}z^j+\sum_{t=1}^{\tau_{s,N}}z^{L_{s,N}+1}\nonumber\\
=&1+s\sum_{j=1}^{L_{s,N}}z^{j}+\tau_{s,N}z^{L_{s,N}+1},
\end{align*}
and so
\[
(1-z)F_{s,N}(z)
=1+(s-1)z+(\tau_{s,N}-s)z^{L_{s,N}+1}
-\tau_{s,N}z^{L_{s,N}+2}
=P_{s,N}(z).
\]

This completes the proof of Lemma \ref{lema335}.
\end{proof}
\begin{lemma}\label{lema35}
Let $s,a,b$ be integers with $1\leq s<a<b$, let
\[
B=\{0,1,\ldots,s,a\},\qquad A=B\cup\{b\},
\]
and write $b=qa+r$, where $0\leq r<a$.  For $h\in\mathbb{N}_0$, let
\[
U_h=\bigcup_{v=0}^{b-1}
\bigl\{tb+v:t\in \Nzero,~~0\leq  t\leq h-\ell_B(v)\bigr\}.
\]
Then
\begin{align*}
\card{U_h}
={}&
\sum_e p_{a,e}
\left\{
\trbinom{h-e+3}{3}
-\trbinom{h-q-e+3}{3}
\right\}
\notag\\
&+\sum_e p_{r,e}
\trbinom{h-q-e+2}{2}
\end{align*}
for any $h\in\mathbb{N}_0$.
\end{lemma}
\begin{proof}
 For a given integer \(v\in[0,b-1]\), the generating function of the corresponding
ray is
\[
\sum_{h\geq0}
\card{\{t:t\in \Nzero,~0\leq t\leq h-\ell_B(v)\}}z^h
=\frac{z^{\ell_B(v)}}{(1-z)^2}.
\]
Hence,
\begin{equation}\label{eq38}
\sum_{h\geq0}\card{U_h}z^h
=
\frac{\sum_{v=0}^{b-1}z^{\ell_B(v)}}{(1-z)^2}.
\end{equation}
For any integer $0\leq u\leq q$ and any integer $0\leq x<a$, by Lemma
\ref{lema33}, we have
\[
\ell_B(ua+x)=u+\ell_B(x).
\]
Therefore,
\begin{align}\label{eq39}
\sum_{v=0}^{b-1}z^{\ell_B(v)}
&=\sum_{u=0}^{q-1}\sum_{x=0}^{a-1}z^{\ell_B(ua+x)}+\sum_{x=0}^{r-1}z^{\ell_B(qa+x)}\notag\\
&=\sum_{u=0}^{q-1}\sum_{x=0}^{a-1}z^{u+\ell_B(x)}+\sum_{x=0}^{r-1}z^{q+\ell_B(x)}\notag\\
&=\frac{1-z^q}{1-z}\sum_{x=0}^{a-1}z^{\ell_B(x)}+z^q\sum_{x=0}^{r-1}z^{\ell_B(x)}.
\end{align}
By Lemmas \ref{lema33} and \ref{lema335}, we have
\[
\sum_{x=0}^{N-1}z^{\ell_B(x)}=\sum_{x=0}^{N-1}z^{\lceil x/s\rceil}=F_{s,N}(z)=\frac{P_{s,N}(z)}{1-z}=\frac{\sum_e p_{N,e}z^e}{1-z}
\]
for any integer $1\leq N\leq a$; for $N=0$, both sides are $0$ by convention.
It follows from \eqref{eq38} and \eqref{eq39} that
\begin{equation}\label{eq310}
\sum_{h\geq0}\card{U_h}z^h
=
\frac{(1-z^q)(\sum_e p_{a,e}z^e)}{(1-z)^4}
+
\frac{z^q(\sum_e p_{r,e}z^e)}{(1-z)^3}.
\end{equation}
Since
\begin{equation*}
\frac{z^\alpha}{(1-z)^{j+1}}
=
\sum_{h\geq0}
\trbinom{h-\alpha+j}{j}z^h
\end{equation*}
for any \(j,\alpha\in\Nzero\),
it follows that the coefficient of \(z^h\) in the first term of
\eqref{eq310} is
\[
\sum_e p_{a,e}
\left\{
\trbinom{h-e+3}{3}
-\trbinom{h-q-e+3}{3}
\right\},
\]
and its coefficient in the second term is
\[
\sum_e p_{r,e}
\trbinom{h-q-e+2}{2}.
\]
Hence,
\begin{align*}
\card{U_h}
={}&
\sum_e p_{a,e}
\left\{
\trbinom{h-e+3}{3}
-\trbinom{h-q-e+3}{3}
\right\}
\notag\\
&+\sum_e p_{r,e}
\trbinom{h-q-e+2}{2}
\end{align*}
for any $h\in\mathbb{N}_0$.

This completes the proof of Lemma \ref{lema35}.

\end{proof}

\begin{proof}[Proof of Theorem \ref{thm2}]
\emph{Sufficiency.}
Let $B=\{0,1,\ldots,s,a\}$.  Then $A=B\cup\{b\}$.  For any
integer $n\geq0$, write $n=ua+x$, where $u\geq0$ and $0\leq x<a$.
It follows from Lemma \ref{lema33} that
\begin{align}\label{eq41}
\ell_B(n+b)-\ell_B(n)
=&\ell_B\left(\left(u+q+\left\lfloor\frac{x+r}{a}\right\rfloor\right)a+[x+r]_a\right)-\ell_B(ua+x)\notag\\
=&\left(u+q+\left\lfloor\frac{x+r}{a}\right\rfloor\right)+\ell_B([x+r]_a)-(u+\ell_B(x))\notag\\
=&q+\left\lfloor\frac{x+r}{a}\right\rfloor+\ell_B([x+r]_a)-\ell_B(x)\notag\\
=&q+\left\lfloor\frac{x+r}{a}\right\rfloor+\left\lceil\frac{[x+r]_a}{s}\right\rceil-\left\lceil \frac{x}{s}\right\rceil.
\end{align}
Since $qa+r=b>a$, we have $q\geq1$.
If $r=0$, then by \eqref{eq41}, we have
\[\ell_B(n+b)-\ell_B(n)=q\geq1.\]
If $qs+r\geq a$, then by \eqref{eq41}, we have
\begin{align}\label{eq42}
\ell_B(n+b)-\ell_B(n)
=&q+\left\lfloor\frac{x+r}{a}\right\rfloor+\left\lceil\frac{[x+r]_a}{s}\right\rceil-\left\lceil \frac{x}{s}\right\rceil\notag\\
\geq&\left\lceil \frac{a-r}{s}\right\rceil+\left\lfloor\frac{x+r}{a}\right\rfloor+\left\lceil\frac{[x+r]_a}{s}\right\rceil-\left\lceil \frac{x}{s}\right\rceil.
\end{align}
If $0\leq x<a-r$, then by \eqref{eq42}, we have
\begin{align*}
\ell_B(n+b)-\ell_B(n)
\geq&\left\lceil \frac{a-r}{s}\right\rceil+\left\lfloor\frac{x+r}{a}\right\rfloor+\left\lceil\frac{[x+r]_a}{s}\right\rceil-\left\lceil \frac{x}{s}\right\rceil\notag\\
=&\left\lceil \frac{a-r}{s}\right\rceil+\left\lceil\frac{x+r}{s}\right\rceil-\left\lceil \frac{x}{s}\right\rceil\notag\\
\geq&\left\lceil \frac{a-r}{s}\right\rceil\geq1.
\end{align*}
If $a-r\leq x<a$, then by \eqref{eq42}, we have
\begin{align*}
\ell_B(n+b)-\ell_B(n)
\geq&\left\lceil \frac{a-r}{s}\right\rceil+\left\lfloor\frac{x+r}{a}\right\rfloor+\left\lceil\frac{[x+r]_a}{s}\right\rceil-\left\lceil \frac{x}{s}\right\rceil\notag\\
=&\left\lceil \frac{a-r}{s}\right\rceil+1+\left\lceil\frac{x-(a-r)}{s}\right\rceil-\left\lceil \frac{x}{s}\right\rceil\notag\\
\geq&1.
\end{align*}
Hence, if $r=0$ or $qs+r\geq a$, then
\begin{equation}\label{eq43}
\ell_B(n+b)-\ell_B(n)\geq1
\end{equation}
for any integer $n\geq0$. For any
\(0\leq v<b\) and any \(t\geq1\), by \eqref{eq43}, we have
\[\ell_B(tb+v)\geq\ell_B((t-1)b+v)+1\geq\cdots\geq \ell_B(v)+t.\]
It follows from Lemma \ref{lema32} that
\[
\ell_A(tb+v)=t+\ell_B(v)
\]
for any integer \(0\leq v<b\) and any integer \(t\geq0\), and so
\[
tb+v\in hA\quad\Longleftrightarrow\quad
\ell_A(tb+v)\leq h
\quad\Longleftrightarrow\quad
0\leq t\leq h-\ell_B(v).
\]
Hence,
\[
hA=\bigcup_{v=0}^{b-1}
\bigl\{tb+v:t\in \Nzero,~0\leq t\leq h-\ell_B(v)\bigr\}=U_h
\]
for any $h\in\mathbb{N}$.
Lemma \ref{lema35} now gives \eqref{eq01}.

\emph{Necessity.}
Suppose that $r>0$ and $qs+r<a$. It is enough to show that
\eqref{eq01} is not true for $h=q+1$. Let
\[
U_{q+1}=\bigcup_{v=0}^{b-1}
\bigl\{tb+v:t\in \Nzero,~0\leq t\leq q+1-\ell_B(v)\bigr\}.
\]
We prove that $U_{q+1}\subsetneqq(q+1)A$.  Let
\[
x=a-r,\qquad d=\left\lceil\frac{x}{s}\right\rceil.
\]
Then $d\geq q+1$ and
\[b+x=qa+r+a-r=(q+1)a\in (q+1)A.\]
If $b+x\in U_{q+1}$, then there exist integers $0\leq v\leq b-1$ and integer
$0\leq t\leq q+1-\ell_B(v)$ such that $b+x=tb+v$.  Since
$0\leq x,v\leq b-1$, uniqueness modulo $b$ gives $v=x$ and $t=1$.
It follows that $1+\ell_B(x)\leq q+1$.  By Lemma \ref{lema33}, we have
\[1+\ell_B(x)=1+\left\lceil\frac{x}{s}\right\rceil=1+d\geq q+2,\]
a contradiction.  Thus $b+x\notin U_{q+1}$.

On the other hand, let $y\in U_{q+1}$.  Then
$y=t_0b+v_0$ for some integers $0\leq v_0\leq b-1$ and
$0\leq t_0\leq q+1-\ell_B(v_0)$. By Lemma \ref{lema31}, we have
\[
\ell_A(y)
=t_0+\min_{0\leq j\leq t_0}\{\ell_B(v_0+jb)-j\}
\leq t_0+\ell_B(v_0)\leq q+1,
\]
and so $y\in(q+1)A$. Therefore,
\[U_{q+1}\subsetneqq(q+1)A,\]
that is $|U_{q+1}|<|(q+1)A|$. By Lemma \ref{lema35},  \eqref{eq01} is not true for $h=q+1$.

This completes the proof of Theorem \ref{thm2}.
\end{proof}

\end{document}